\documentclass[12pt,a4paper,oneside]{amsart}
\pdfoutput=1
\usepackage{amsmath,amsthm}
\usepackage{esint}
\usepackage{mathtools}
\usepackage{mathrsfs}
\usepackage{fullpage}
\usepackage{hyperref}
\usepackage{amsmath,amscd}
\usepackage{amssymb}
\usepackage[english]{babel}
\usepackage[utf8]{inputenc}
\usepackage[pdftex]{graphicx}
\theoremstyle{plain}
\newtheorem{theorem}{Theorem}
\newtheorem{lemma}[theorem]{Lemma}

\newtheorem{proposition}[theorem]{Proposition}

\theoremstyle{definition}
\newtheorem{definition}[theorem]{Definition}

\theoremstyle{remark}

\newcommand{\R}{\mathbb R}

\newcommand{\N}{\mathbb N}

\newcommand{\der}{\mathrm{d}}
\newcommand{\eps}{\varepsilon}
\renewcommand{\phi}{\varphi}
\newcommand{\abs}[1]{\left\lvert #1 \right\rvert}
\newcommand{\aabs}[1]{\left\lVert #1 \right\rVert}
\newcommand{\sisus}{\operatorname{int}}

\renewcommand{\theta}{\vartheta}

\newcommand{\ip}[2]{\left\langle#1,#2\right\rangle}
\newcommand{\iip}[2]{\left(#1,#2\right)}

\newcommand{\Der}[1]{\frac{\der}{\der #1}}
\newcommand{\sd}{\der^s}
\DeclareMathOperator{\Lip}{Lip}
\newcommand{\dummy}{\,\cdot\,}
\newcommand{\pSM}{\partial(SM)}

\newcommand{\reflect}{\rho}
\newcommand{\reverse}{\mathcal{R}}
\def \vd{\overset{\smash{\tt{v}}}{\nabla}}
\def \hd{\overset{\smash{\tt{h}}}{\nabla}}
\def \ehd{\overset{\smash{\tt{h}}}{\overline{\nabla}}}
\def \vdiv{\overset{\smash{\tt{v}}}{\mbox{\rm div}}}
\def \hdiv{\overset{\smash{\tt{h}}}{\mbox{\rm div}}}

\newcommand{\norm}[1]{\aabs{#1}}
\newcommand{\vnu}{\ip{v}{\nu}}

\newcommand{\vgrad}{\vd}
\newcommand{\hgrad}{\hd}
\newcommand{\ehgrad}{\ehd}

\newcommand{\sdot}[1]{\vphantom{#1}\smash{\dot{#1}}}

\newcommand{\outerbdy}{{\mathscr E}}
\newcommand{\innerbdy}{{\mathscr R}}

\newcommand{\ntr}[1]{}

\mathtoolsset{showonlyrefs}

\title[Broken ray tensor tomography]{Broken ray tensor tomography\\with one reflecting obstacle}
\author{Joonas Ilmavirta}
\thanks{Department of Mathematics and Statistics, University of Jyv\"askyl\"a, Finland. \texttt{joonas.ilmavirta@jyu.fi}}
\author{Gabriel P. Paternain}
\thanks{Department of Pure Mathematics and Mathematical Statistics, University of Cambridge, UK. \texttt{g.p.paternain@dpmms.cam.ac.uk}}
\date{\today}

\begin{document}

\begin{abstract}
We show that a tensor field of any rank integrates to zero over all broken rays if and only if it is a symmetrized covariant derivative of a lower order tensor which satisfies a symmetry condition at the reflecting part of the boundary and vanishes on the rest. This is done in a geometry with non-positive sectional curvature and a strictly convex obstacle in any dimension. We give two proofs, both of which contain new features also in the absence of reflections. The result is new even for scalars in dimensions above two.
\end{abstract}



\maketitle

\section{Introduction}

\ntr{We have indicated all changes in the text with footnotes like this.}

We study the problem of unique determination of a tensor field from its integrals over all broken rays on a Riemannian manifold.
When broken geodesic rays are replaced with unbroken ones, this is a classical problem and we refer the reader to the review~\cite{PSU:review}.
Analogously to the X-ray transform, the broken ray transform of a function or a tensor field on a manifold is a function on the space of broken rays defined by integration over broken rays.
The broken rays have endpoints on a subset $\outerbdy\subset\partial M$ and reflect on $\innerbdy\subset\partial M$.
\ntr{Added two sentences to define the broken ray transform.}
We show injectivity up to natural gauge obstructions on a compact non-positively curved manifold with dimension $n\geq2$ with a strictly convex boundary and a strictly convex reflecting obstacle.
A tensor field~$f$ of order~$m$ has vanishing broken ray transform if and only if there is a tensor field~$h$ of order $m-1$ so that $f=\sd h$ (symmetrized covariant derivative of~$h$) and $h$ satisfies a reflection condition at the surface of the reflector.

This result for $n=2$ and $m=0$ was proved in~\cite{IS:brt}.
The results for tensor fields are new, as is the injectivity for scalar functions in dimension three and higher.
Under stringent assumptions on several reflecting\ntr{Fixed spelling.} obstacles in the Euclidean space, Eskin~\cite{E:brt} showed this result for $n=2$ and $m\in\{0,1\}$.
Sharafutdinov~\cite{S:tt-annulus} showed solenoidal injectivity for the X-ray transform for any~$m$ in an annulus with rotation invariant Riemannian metric.
If in addition to geodesics avoiding the obstacle one uses the broken rays that reflect on it, one ends up with a restriction on the gauge condition on the surface of the reflector.

Broken ray tomography of scalar fields has been studied more extensively; see~\cite{I:phd}.
The methods used to prove injectivity are explicit calculation in spherically symmetric geometries~\cite{dHI:Abel,I:disk}, using a reflection argument to reduce it to X-ray tomography~\cite{H:square,H:cube,I:refl}, and applying a Pestov identity~\cite{E:brt,IS:brt}.
Boundary determination for broken ray tomography with concave reflectors leads to weighted X-ray tomography on the boundary manifold~\cite{I:bdy-det}.

The broken ray transform for scalars and one-forms is related to inverse boundary value problems for PDEs~\cite{KS:Calderon,E:brt}.
The broken ray transform of two-tensors arises from the linearization of length of broken rays~\cite{IS:brt}.
X-ray transforms are ubiquitous in various inverse problems in analysis and geometry, often arising through linearization or special asymptotic solutions to PDEs.
When the underlying problem has only partial data, we expect that in many applications the X-ray transform is replaced with a broken ray transform with reflections at the inaccessible part of the boundary.

Perhaps the most important example of a three-dimensional object with a reflecting obstacle inside is the Earth.
Seismic waves reflect on the various interfaces, the most pronounced of them being the core--mantle boundary.
As reflections are an inevitable aspect of the physical setting, better understanding of broken ray tomography is a contribution to the theory of seismic imaging.

In broad terms, our approach is based on ideas first put forward by Guillemin and Kazhdan~\cite{GK:2,GK:n}.
We use energy estimates known as Pestov identities and methods related to Beurling transforms to obtain tensor tomography results as in~\cite{PSU:Beurling,LRS:CH}.
The main difference is two-fold:
the Pestov identity contains an additional boundary term
and
the integral function defined on the sphere bundle is not smooth a priori.
The first issue is due to the integral function not vanishing on the surface of the reflector and the second one due to non-smoothness of the broken geodesic flow and Jacobi fields when the ray hits the reflector tangentially.
Singular Jacobi fields have also been studied in other contexts; see e.g.~\cite{GGP:Morse}.

We provide an alternative proof of our theorem using a different argument.
It is also based on the Pestov identity, but the identity is used in a very different way.
The first proof is based on applying the identity to each component of the integral function~$u$ of~$f$ separately and showing certain positivity due to negative curvature of both the manifold and the reflecting obstacle.
The resulting estimate will given an estimate for the behaviour of the tail of the spherical harmonic expansion and show that~$u$ has the correct order.
In the second proof we cut of the beginning of the spherical harmonic expansion of~$u$ and make use of the way differential operators and their commutators map between different spherical harmonic orders.
We apply the resulting estimate to this truncated~$u$ and conclude that it vanishes and thus the original~$u$ has the correct order.
\ntr{Added a description and comparison of the two methods.}
This argument also gives a concise proof of non-existence of trace-free conformal Killing tensors.

\subsection{Notation}

We review the notation needed to state our main results.
Let~$M$ be a Riemannian manifold with boundary.
We denote the unit sphere bundle by~$SM$.

We define the reversion map $\reverse\colon SM\to SM$ by $\reverse(x,v)=(x,-v)$ and the reflection map $\reflect\colon\pSM\to\pSM$ by $\reflect(x,v)=(x,v-2\ip{v}{\nu}\nu)$, where $\nu=\nu(x)$ is the outer unit normal to~$\partial M$ at~$x$.
Both~$\reverse$ and~$\reflect$ are involutions, and they commute on~$\pSM$.

A broken ray on a manifold with boundary is a geodesic which reflects at the boundary.
The reflections are defined so that the incoming and outgoing directions are related by~$\reflect$, which in dimension two amounts to saying that the angle of incidence equals the angle of reflection.
All our geodesics and broken rays have unit speed.

The integral of a symmetric covariant tensor field over a broken ray is defined in the usual way; see equation~\eqref{eq:tensor-SM}.
We denote the symmetrized covariant derivative of such a tensor field by~$\sd$.

\begin{definition}
\label{def:admissible}
Let $(M,g)$ be a smooth Riemannian manifold with smooth boundary so that~$\partial M$ is a disjoint union of relatively open sets~$\outerbdy$ and $\partial M\setminus\outerbdy\eqqcolon\innerbdy$.
The triplet $(M,g,\outerbdy)$ is called \emph{admissible} if the following hold:
\begin{enumerate}
\item
$M$ is compact.
\item
The boundary is strictly convex at~$\outerbdy$ and strictly concave at~$\innerbdy$ in the sense of the second fundamental form.
\item
The sectional curvature of $(M,g)$ is non-positive.
\item
There are $L>0$ and $a>0$ so that for any given point $(x,v)\in SM$ the broken ray starting at~$x$ in direction~$v$ reaches~$\outerbdy$ in time bounded by~$L$ and has at most one reflection at~$\innerbdy$ with $\abs{\ip{\nu}{\dot\gamma}}<a$.
\end{enumerate}
\end{definition}

The last condition implies that a broken ray can have at most one tangential reflection in a strong sense.
See~\cite[Remark 3]{IS:brt} for a discussion on this condition.

The simplest example is a simply connected non-positively curved manifold with strictly convex boundary and $\outerbdy=\partial M$.
Then there are no reflections and we are left with the usual tensor tomography problem.
However, even in this case our method of proof contains new ideas.

A more interesting example is obtained when one adds a reflecting obstacle to the simply connected non-positively curved manifold with strictly convex boundary.
If the obstacle is strictly convex and no broken ray hits it twice, the resulting manifold is admissible.
Even the case of a strictly convex obstacle in a Euclidean space is new.

\subsection{The main result}

Our main result is the following solenoidal injectivity theorem for the broken ray transform.

\begin{theorem}
\label{thm:tensor}
Let $(M,g,\outerbdy)$ be admissible in the sense of definition~\ref{def:admissible}.
Assume $n\coloneqq\dim(M)\geq2$.
Then the broken ray transform is solenoidally injective on tensor fields in the following sense:
a $C^2$-regular symmetric covariant tensor field~$f$ of order $m\geq0$ integrates to zero over all broken rays with endpoints on $\outerbdy$ if and only if $f=\sd h$ for a symmetric covariant tensor field~$h$ of order $m-1$ which satisfies $h=0$ at $\outerbdy$ and $h=h\circ\reflect$ at $\innerbdy$.
In particular, a scalar field ($m=0$) integrates to zero over all broken rays if and only if it vanishes identically.
\end{theorem}

The case with $n=2$ and $m=0$ was covered in~\cite{IS:brt}.
In Euclidean geometry uniqueness can be proven for $m=0$ using Helgason's support theorem using only lines that do not hit the obstacle.
For $m\geq1$ one may use Sharafutdinov's result~\cite{S:tt-annulus} around any ball and argue that if~$f$ integrates to zero over every geodesic avoiding the obstacle, then there is a lower order tensor field~$h$ defined outside the obstacle so that $f=\sd h$, and computation by hand can be used to find additional conditions at the surface of the reflector when broken rays are added to the data.
However, we are not aware of the broken ray tensor tomography result being explicitly stated in Euclidean geometry.

The reflection $\reflect\colon\pSM\to\pSM$ can be naturally extended to a map $\partial(TM)\to\partial(TM)$, giving a way to see $h\mapsto h\colon\reflect$ as a transformation of multilinear maps on~$T_xM$ with $x\in\innerbdy$.
This coincides with the point of view given by seeing a tensor field as a function of finite degree on the sphere bundle.
\ntr{Added a discussion of the reflection condition.}
The reflection condition on~$\innerbdy$ is vacuous for $m=1$ as~$h$ is scalar.
When $m=2$, it says that the one-form~$h$ is tangential to the boundary.
In general, the reflection condition can be seen as extendability:
if two copies of the manifold~$M$ are glued together at~$\innerbdy$, then a tensor field~$h$ on~$M$ becomes a tensor field on the doubled manifold if and only if it satisfies the reflection condition of theorem~\ref{thm:tensor}.
We do not employ the doubling method, but reflection arguments have been used successfully for broken ray tomography as mentioned above.

Consider the non-linear problem of determining a Riemannian manifold from the lengths of all broken rays.
As mentioned earlier, the linearized version of the problem is to recover a rank two tensor field from its integrals over all broken rays.
Both problems have a gauge freedom:
the non-linear problem is invariant under changes of coordinates and the linear one under addition of potentials (symmetrized derivatives of one-forms).
The one-forms can be regarded as an infinitesimal generator of the diffeomorphisms to change coordinates.
The boundary conditions on the one-form are as follows:
at~$\outerbdy$ it vanishes (the diffeomorphism fixes every point on~$\outerbdy$), and at~$\innerbdy$ it is tangential to the boundary (the diffeomorphism fixes the set~$\innerbdy$ but not necessarily every point on it).

\subsection{Structure of the article}

The necessary tools and concepts required for the proof of the theorem are given in section~\ref{sec:prelim-and-proofs}, and also the theorem is proven there.
The following sections are for providing proofs of the various lemmas needed in the proofs:
section~\ref{sec:reg} establishes regularity results, section~\ref{sec:pestov} the Pestov identity with boundary terms, and section~\ref{sec:X+-} some mapping properties of the operators $X_\pm$ defined below.
In section~\ref{sec:alt} we give an alternative proof of our result and give results on conformal Killing tensors in section~\ref{sec:killing}.

\section{Preliminaries and proofs of theorems}
\label{sec:prelim-and-proofs}

\subsection{Operators and decompositions on the sphere bundle}

We mostly follow the presentation of~\cite{PSU:Beurling} for the basic structure of the sphere bundle.

Let $(M,g)$ be a Riemannian manifold with unit sphere bundle $\pi\colon SM\to M$ and as always let~$X$ be the geodesic vector field.
It is well known that~$SM$ carries a canonical metric called the Sasaki metric.
If we let~$\mathcal V$ denote the vertical subbundle given by $\mathcal V=\ker(\der\pi)$, then there is an orthogonal splitting with respect to the Sasaki metric:
\begin{equation}
TSM
=
\R X\oplus {\mathcal H}\oplus {\mathcal V}.
\end{equation}
The subbundle~${\mathcal H}$ is called the horizontal subbundle.
Elements in $\mathcal H(x,v)$ and $\mathcal V(x,v)$ are canonically identified with elements in the codimension one subspace $\{v\}^{\perp}\subset T_{x}M$.
We shall use this identification freely below.

Given a smooth function $u\in C^{\infty}(SM)$ we can consider its gradient~$\nabla u$ with respect to the Sasaki metric.
Using the splitting
above we may write uniquely
\begin{equation}
\nabla u=((Xu)X,\hd u,  \vd u).
\end{equation}
The derivatives~$\hd u$ and~$\vd u$ are called horizontal and vertical gradients respectively. 

We shall denote by~$\mathcal Z$ the set of smooth functions $Z\colon SM\to TM$ such that $Z(x,v)\in T_{x}M$ and $\ip{Z(x,v)}{v}=0$ for all $(x,v)\in SM$.
With the identification mentioned above we see that $\hd u,\vd u\in \mathcal Z$.

The geodesic vector field~$X$ acts on~$\mathcal Z$ by
\begin{equation}
XZ(x,v)
=
\left.\frac{DZ(\phi_{t}(x,v))}{\der t}\right|_{t=0}
\end{equation}
where~$\phi_t$ is the geodesic flow and $Z\in\mathcal Z$.
Note that $Z(t)\coloneqq Z(\phi_{t}(x,v))$ is a vector field along the geodesic~$\gamma$ determined by $(x,v)$, so it makes sense to take its covariant derivative with respect to the Levi--Civita connection of~$M$.
Since $\ip{Z}{\dot{\gamma}}=0$ it follows that
$\ip{\frac{DZ}{dt}}{\dot{\gamma}}=0$ and hence $XZ\in \mathcal Z$.

Another way to describe the elements of~$\mathcal Z$ is a follows.
Consider the pull-back bundle $\pi^*TM\to SM$.
Let~$N$ denote the subbundle of~$\pi^*TM$ whose fiber over $(x,v)$ is given by $N_{(x,v)}=\{v\}^{\perp}$.
Then~$\mathcal Z$ coincides with the smooth sections of the bundle~$N$.
Observe that~$N$ carries a natural~$L^{2}$ inner product and with respect to this product the formal adjoints of $\vd\colon C^{\infty}(SM)\to\mathcal Z$ and $\hd\colon C^{\infty}(SM) \to \mathcal Z$ are denoted by $-\vdiv$ and $-\hdiv$ respectively.
Note that since~$X$ leaves invariant the volume form of the Sasaki metric we have $X^*=-X$ for both actions of~$X$ on~$C^{\infty}(SM)$ and~$\mathcal Z$.

Let $R(x,v)\colon\{v\}^{\perp}\to \{v\}^{\perp}$ be the operator determined by the Riemann curvature tensor $R$ by $R(x,v)w=R_{x}(w,v)v$ and let $n=\dim M$.
We will also make use of the total horizontal gradient $\ehd u(x,v)=vXu(x,v)+\hgrad u(x,v)\in T_xM$.

These operators satisfy the following commutator formulas:
\begin{equation}
\begin{split}
[X,\vd]&=-\hd, \\ 
[X,\hd]&=R\vd, \\ 
\hdiv\vd-\vdiv\hd&=(n-1)X. 
\end{split}
\end{equation}
Taking adjoints gives the following commutator formulas on~$\mathcal Z$:
\begin{equation}
\begin{split}
[X,\vdiv] &= -\hdiv, \\
[X,\hdiv] &= -\vdiv R.
\end{split}
\end{equation}
More commutator formulas may be derived from these, including~\cite[lemma 3.5]{PSU:Beurling}
\begin{equation}
\label{eq:[X,D]}
[X,\Delta]
=
2\vdiv\hgrad+(n-1)X,
\end{equation}
where~$\Delta$ is the vertical Laplacian to which we shall return shortly.
\ntr{Reworded.}

The boundary of the sphere bundle is the disjoint union $\pSM=\partial_+(SM)\cup\partial_-(SM)\cup\partial_0(SM)$, where
\begin{equation}
\begin{split}
\partial_\pm SM
&=
\{(x,v)\in SM;\pm\vnu>0\}
\quad\text{and}
\\
\partial_0 SM
&=
\{(x,v)\in SM;\vnu=0\}.
\end{split}
\end{equation}
Here~$\nu$ is the outer unit normal to the boundary, so~$\partial_-(SM)$ is the inward-pointing boundary.

We will need the integration\ntr{Fixed spelling.} by parts formulas
\begin{equation}
\begin{split}
\iip{\vd u}{Z}&=-\iip{u}{\vdiv Z}, \\ 
\iip{Xu}{w}&=-\iip{u}{Xw}+\iip{\vnu u}{w}_{\pSM}, \quad\text{and} \\ 
\iip{XZ}{W}&=-\iip{Z}{XW}+\iip{\vnu Z}{W}_{\pSM} 
\end{split}
\end{equation}
for $u,w\in C^{\infty}(SM)$ and $Z,W\in \mathcal Z$.
The convention is as follows:
when there is no subscript the norms and inner products are in~$L^{2}(SM)$, and the ones for~$L^2(\pSM)$ are marked.

On every fiber we may decompose a function on~$S_xM$ into the eigenspaces of the vertical Laplacian $\Delta=-\vdiv\vgrad$:
\begin{equation}
L^2(S_xM)
=
\bigoplus_{k=0}^\infty \Lambda^k_x,
\end{equation}
where
\begin{equation}
\Lambda^k_x
=
\{u\colon S_xM\to\R;\Delta u=k(k+n-2)u\}.
\end{equation}
This gives rise to a decomposition on the whole sphere bundle:
\begin{equation}
L^2(SM)
=
\bigoplus_{k=0}^\infty \Lambda^k,
\end{equation}
where $\Lambda^k$ is the set of functions $u\in L^2(SM)$ for which $u=\tilde u$ almost everywhere and for every $x\in M$ the function $\tilde u$ satisfies $\tilde u(x,\dummy)\in\Lambda^k_x$.
Functions in~$\Lambda^k\subset L^2(SM)$ are referred to as functions of degree~$k$.
A function in~$L^2(SM)$ is said to have finite degree if it only contains components in finitely many of the spaces~$\Lambda^k$.
Using the eigenvalue property shows that
\begin{equation}
\label{eq:vgrad-uk}
\aabs{\vgrad u}^2
=
k(k+n-2)\aabs{u}^2
\end{equation}
for a function~$u$ of degree~$k$.

The geodesic vector field~$X$ may be decomposed as $X=X_++X_-$, where~$X_\pm$ maps functions of degree~$k$ to functions of degree $k\pm1$.
Establishing mapping properties for~$X_\pm$ is a crucial ingredient in our proof.
This decomposition is due to~\cite{GK:n}.

A symmetric covariant tensor field $f$ of order $m\geq0$ can be regarded as a function~$\tilde f$ on~$SM$ by letting
\begin{equation}
\label{eq:tensor-SM}
\tilde f(x,v)
=
f_x(v,\dots,v).
\end{equation}
We will freely identify~$f$ and~$\tilde f$.
The function on the sphere bundle corresponding to a tensor field of order $m$ contains only degrees~$m$, $m-2$, $m-4$, \dots, and any of these different degree components may vanish.
For example, the metric tensor has rank~$2$, but the corresponding function on the sphere bundle is constant and is therefore of degree~$0$.

The most important operator for symmetric covariant tensor fields is the symmetrized covariant derivative~$\sd$, which appears in the gauge condition for tensor tomography.
When the tensor fields are identified with functions on~$SM$, the derivative~$\sd$ becomes the geodesic vector field~$X$.
Checking this is a straightforward computation, and one can use geodesic normal coordinates at any point of interest to essentially reduce the problem to its Euclidean counterpart.
For more details, consult e.g.~\cite[Lemma 10.1]{DS:Killing}.

A piecewise~$C^1$ curve~$\gamma$ on~$M$ parametrized by arc length\ntr{Included arc length parametrization.} may be lifted to a curve~$\sigma$ on~$SM$ by $\sigma(t)=(\gamma(t),\dot\gamma(t))$.
The integral of a tensor field or any other function on~$SM$ over a geodesic or a broken ray is defined to be the integral over the lifted curve.

We will make use of the reversion operator $\reverse\colon SM\to SM$ defined by $\reverse(x,v)=(x,-v)$ and the reflection operator $\reflect\colon\pSM\to\pSM$ defined by $\reflect(x,v)=(x,v-2\vnu\nu)$.
We will denote the restriction of~$\reflect$ to~$S_xM$ for a fixed $x\in\partial M$ by~$\reflect_x$.
Decomposition of functions into even and odd parts with respect to~$\reverse$ and~$\reflect$ will be convenient.

\subsection{Proof of theorem~\ref{thm:tensor}}

We will now prove theorem~\ref{thm:tensor}.
The lemmas of this section will be proved later.

Let~$f$ be a tensor field which integrates to zero over all broken rays with endpoints on~$\outerbdy$.
We will show that it is of the desired form.
The converse statement follows by applying the fundamental theorem of calculus along every geodesic segment of any given broken ray.
We will consider~$f$ as a function $SM\to\R$ as explained above.

For $(x,v)\in SM$, we denote by $\gamma_{x,v}\colon[0,\tau(x,v)]\to M$ the broken ray starting at~$x$ in direction~$v$ so that $\gamma_{x,v}(\tau(x,v))\in\outerbdy$.
\ntr{Rewrote the sentence so that we integrate over broken rays instead of geodesics.}
We define a function $u\colon SM\to\R$ by
\begin{equation}
\label{eq:u-def}
u(x,v)
=
\int_0^{\tau(x,v)}f(\gamma_{x,v}(t),\dot\gamma_{x,v}(t))\der t.
\end{equation}
It follows from the admissibility assumption that $\tau\leq L$, and therefore~$u$ is pointwise well defined.
We first need to establish some regularity for~$u$.
It is immediate that~$u$ is differentiable along the geodesic flow and $Xu=-f$.

\begin{lemma}
\label{lma:u-reg}
Suppose $f\in C^2(SM)$ and define~$u$ by equation~\eqref{eq:u-def}.
If~$M$ is admissible and~$u$ vanishes at~$\outerbdy$, then $u\in C^2(\sisus SM)\cap\Lip(SM)$.
In addition, $u=u\circ\reflect$ at $\innerbdy$.
\end{lemma}

\begin{lemma}
\label{lma:uk-reg}
In the setting of lemma~\ref{lma:u-reg}, let $u=\sum_ku_k$ be the spherical harmonic decomposition of~$u$.
Then $u_k\in C^2(\sisus SM)\cap\Lip(SM)$ and $u_k=u_k\circ\reflect$ at~$\innerbdy$ for every $k\in\N$.
\end{lemma}

To gain better control of regularity, we need to understand the properties of~$X_\pm$.

\begin{lemma}
\label{lma:X+-u-L2}
Assume $(M,g,\outerbdy)$ is admissible.
Let~$f$ be a $C^2$-regular\ntr{Added regularity assumption.} tensor field of order~$m$ with vanishing broken ray transform and~$u$ as defined in~\eqref{eq:u-def}.
Then $X_+u,X_-u\in L^2(SM)$.
\end{lemma}

\begin{lemma}
\label{lma:X+inj}
Assume $(M,g)$ is a smooth, compact, and connected Riemannian manifold with smooth boundary.
Suppose $u\in C^2(\sisus SM)\cap\Lip(SM)$ has degree $k\geq3$.
If $X_+u=0$ and~$u$ vanishes on a non-empty open subset of~$\partial M$, then $u=0$.
\end{lemma}

Another version of lemma~\ref{lma:X+inj} is given in proposition~\ref{prop:killing1}.
The limit on degree in the lemma above is merely a matter of convenience; the result is only used for degrees~$3$ and higher.

\begin{lemma}
\label{lma:L2-reg}
Assume $(M,g,\outerbdy)$ is admissible.
Let~$f$ be a $C^2$-regular\ntr{Added regularity assumption.} tensor field of order~$m$ with vanishing broken ray transform and~$u$ as defined in~\eqref{eq:u-def}.
Then $\vgrad Xu\in L^2(SM)$ and $\vgrad Xu_k\in L^2(SM)$ for every $k\in\N$.
\end{lemma}

Now that~$u$ and~$u_k$ are sufficiently regular, we may apply a Pestov identity.
The identity contains a boundary term featuring the differential operator
\begin{equation}
\label{eq:Q-def}
Q\coloneqq \vnu\ehgrad-\nu X
\end{equation}
defined at the boundary of the sphere bundle.
This operator will be discussed in section~\ref{sec:bdy-terms}.

\begin{lemma}
\label{lma:pestov-smooth}
Let~$M$ be a smooth and compact Riemannian manifold with smooth boundary.
If $u\in C^2(SM)$, then
\begin{equation}
\aabs{\vgrad Xu}^2
=
\aabs{X\vgrad u}^2
-\iip{R\vgrad u}{\vgrad u}
+(n-1)\aabs{Xu}^2
+P(u,u),
\end{equation}
where~$P$ is the quadratic form defined by
\begin{equation}
P(u,w)
=
\iip{Qu}{\vgrad w}_{\pSM}.
\end{equation}
\end{lemma}

The boundary term has a special form when~$u$ has a reflection symmetry at~$\innerbdy$.
To this end, let us define a new quadratic form~$H$ by
\begin{equation}
H(u,w)
=
\int_{\pSM}\Pi_{x}(Tu(x,v),Tw(x,v))\der\Sigma^{2n-2},
\end{equation}
where
\begin{equation}
Tu(x,v)
\coloneqq
\langle \vd u,\nu\rangle v-\vnu\vd u
\end{equation}
and~$\Pi_x$ is the second fundamental form at $x\in\partial M$.
Geometrically, $Tu(x,v)$ is a projection of $\vgrad u\in T_xM$ to $T_x(\partial M)$.
For the different boundary components~$\outerbdy$ and~$\innerbdy$ we may naturally define~$P_\outerbdy$, $P_\innerbdy$, $H_\outerbdy$, and~$H_\innerbdy$ by restricting the quadratic form to the relevant component.

\begin{lemma}
\label{lma:bdy-term}
Let~$M$ be a smooth and compact Riemannian manifold with smooth boundary.
If $u\in C^2(SM)$ and $u\circ\reflect=\pm u$ at~$\partial M$, then $P(u,u)=-H(u,u)$.
\end{lemma}

An approximation argument is needed to prove a Pestov identity for the regularity and boundary behavior provided by lemmas~\ref{lma:u-reg} and~\ref{lma:uk-reg}.

\begin{lemma}
\label{lma:pestov-lip}
Let $(M,g,\outerbdy)$ be admissible.
If $u\in C^2(\sisus SM)\cap\Lip(SM)$, $\vgrad Xu\in L^2(SM)$, $u=u\circ\reflect$ at $\innerbdy$ and $u=0$ at~$\outerbdy$, then $X\vgrad u\in L^2(SM)$ and
\begin{equation}
\label{eq:pestov}
\aabs{\vgrad Xu}^2
=
\aabs{X\vgrad u}^2
-\iip{R\vgrad u}{\vgrad u}
+(n-1)\aabs{Xu}^2
-H_\innerbdy(u,u).
\end{equation}
In particular,
\begin{equation}
\label{eq:pestov-estimate}
\aabs{\vgrad Xu}^2
\geq
\aabs{X\vgrad u}^2
+(n-1)\aabs{Xu}^2.
\end{equation}
\end{lemma}

The boundary term of~\eqref{eq:pestov} reduces to the two-dimensional one obtained in~\cite{IS:brt} following the comparison of the general framework and the two-dimensional one given in~\cite[Appendix B]{PSU:Beurling}.

Lemma~\ref{lma:pestov-lip} is enough to prove the theorem for $m=0$ and $m=1$.
If $m=0$, then $Xu=-f$ is constant on each fiber.
Therefore $\vgrad Xu=0$ and~\eqref{eq:pestov-estimate} implies $\aabs{Xu}=0$, from which we conclude that $f=-Xu=0$ as desired.

If $m=1$, then $Xu=-f$ has rank one.
Using~\eqref{eq:vgrad-uk} with $k=1$ gives $\aabs{\vgrad Xu}^2=(n-1)\aabs{Xu}^2$, and we may conclude $X\vgrad u=0$.
This means\ntr{Removed repeated word.} that~$\vgrad u$ is constant along the geodesic flow, and the reflection condition ensures that it is also constant along the broken geodesic flow.
Since $u$ vanishes at~$\outerbdy$ and all broken rays reach~$\outerbdy$ in finite time, the derivative~$\vgrad u$ has to vanish identically on~$SM$.
This means that $u=-\pi^*h$ for some scalar function~$h$.
Now the condition $Xu=-f$ means that $f=\der h$, which proves solenoidal injectivity.

For $m\geq2$ more tools are needed.
Instead of working with the whole function~$u$, we study the terms~$u_k$ separately.
Comparing these terms will involve the constants
\begin{equation}
C(k,n)
=
1+\frac{2}{2k+n-3}
\end{equation}
and
\begin{equation}
B(k,N,n)
=
\prod_{l=1}^N C(k+2l,n).
\end{equation}
Studying the transport equation quickly gives an\ntr{Fixed spelling error.} identity connecting~$X_-$ and~$X_+$.

\begin{lemma}
\label{lma:X+=X-}
Assume $(M,g,\outerbdy)$ is admissible.
Let~$f$ be a tensor field of order~$m$ with vanishing broken ray transform and~$u$ as defined in~\eqref{eq:u-def}.
If $k\geq m$ or $k-m$ is even, then $\aabs{X_+u_k}^2=\aabs{X_-u_{k+2}}^2$.
\end{lemma}

Applying the Pestov identity of lemma~\ref{lma:pestov-lip} to~$u_k$ gives an estimate between~$X_-$ and~$X_+$, which may be seen as continuity of the Beurling transform.

\begin{lemma}
\label{lma:X-<X+}
Assume $(M,g,\outerbdy)$ is admissible.
Let $f$ be a tensor field of order $m$ with vanishing broken ray transform and $u$ as defined in~\eqref{eq:u-def}.
Then we have
\begin{equation}
\aabs{X_-u_k}^2
\leq
C(k,n)
\aabs{X_+u_k}^2
\end{equation}
whenever $2k+n>3$.
\end{lemma}

Finally, we need an estimate for the constants.
Our constants $C(k,n)$ are not optimal for lemma~\ref{lma:X-<X+}, but are sufficient for our proof.
More detailed estimates have been used~\cite{PSU:Beurling} to show that $\lim_{N\to\infty}B(k,N,n)$ (with~$B$ redefined with sharper constants~$C$) is finite.

\begin{lemma}
\label{lma:prod-est}
The constant $B(k,N,n)$ defined above satisfies
\begin{equation}
B(k,N,n)
\leq
\sqrt{1+\frac{4N}{2k+n-3}}
\end{equation}
whenever $2k+n>3$.
\end{lemma}

Now we are ready to finally prove the theorem.

\begin{proof}[Proof of theorem~\ref{thm:tensor}]
We gave a short proof for $m=0$ and $m=1$ above.
Therefore we assume $m\geq2$, although most of the arguments do not rely on this.

Define~$u$ as in~\eqref{eq:u-def}.
By lemmas~\ref{lma:u-reg} and~\ref{lma:uk-reg} both~$u$ and each~$u_k$ have sufficient regularity to apply lemma~\ref{lma:pestov-lip} to obtain a Pestov identity.
Using this identity for~$u_k$ leads to lemma~\ref{lma:X-<X+}.

The function~$f$ satisfies $f\circ\reverse=(-1)^mf$.
Since~$f$ integrates to zero over all broken rays, it follows from the definition of~$u$ that $u\circ\reverse=(-1)^{m+1}u$.
By symmetry of spherical harmonics we have $u_k\circ\reverse=(-1)^ku_k$.
\ntr{Added remark on symmetry of spherical harmonics to make the conclusion clearer.}
Therefore $u_k=0$ whenever $k-m$ is even.

Let $m_0\in\N$ be such that  $m_0\geq m$ and $m_0-m$ is odd.
Notice that $m_0\geq3$ by our assumption $m\geq2$.
Then combining lemmas~\ref{lma:X+=X-} and~\ref{lma:X-<X+} $k$ times gives
\begin{equation}
\label{eq:iterated-beurling}
\aabs{X_+u_{m_0}}^2
\leq
B(m_0,k,n)
\aabs{X_+u_{m_0+2k}}^2.
\end{equation}
Suppose $\aabs{X_+u_{m_0}}^2=a>0$.
Then~\eqref{eq:iterated-beurling} yields
\ntr{Fixed a wrong index in the formula.}
\begin{equation}
\aabs{X_+u_{m_0+2k}}^2
\geq
a
B(m_0,k,n)^{-1}.
\end{equation}
Lemma~\ref{lma:prod-est} gives\ntr{Fixed the index, same as above.} $B(m_0,k,n)^{-1}\geq Ak^{-1/2}$ for some constant $A>0$ depending on~$m_0$ and~$n$; the assumption $2m_0+n>3$ is always satisfied for $m_0\geq3$.
Thus
\begin{equation}
\sum_{k=1}^\infty\aabs{X_+u_{m_0+2k}}^2
\geq
aA\sum_{k=1}^\infty k^{-1/2}
=
\infty.
\end{equation}
But by lemma~\ref{lma:X+-u-L2} we have
\begin{equation}
\aabs{X_+u}^2
=
\sum_{k=0}^\infty\aabs{X_+u_k}^2
<
\infty.
\end{equation}
This is a contradiction, so we must have $a=0$.

We conclude that when $k\geq m+2$ and $k-m$ is odd, we have $X_+u_k=0$.
Since~$u_k$ vanishes at~$\outerbdy$, it follows from lemma~\ref{lma:X+inj} that $u_k=0$.

Therefore~$u$ arises from a tensor field $-h$ of order $m-1$.
The transport equation $Xu=-f$ implies $\sd h=f$.
This finally proves the claim.
\end{proof}

Our method of proof is different from those used before.
In~\cite{PSU:Beurling,LRS:CH} sharper estimates are obtained and therefore the products $B(k,N,n)$ of constants are uniformly bounded.
Then if one argues that $\aabs{X_+u_k}\to0$ as $k\to0$, one can reach the same conclusion that $a=0$.
We used simpler estimates at the expense of not having a uniform bound for $B(k,N,n)$.
This simplification is possible because $X_+u\in L^2(SM)$ implies more than just $\aabs{X_+u_k}\to0$.
Another new proof is presented in section~\ref{sec:alt}.

If we do not appeal to injectivity of~$X_+$, we may use lemma~\ref{lma:X+=X-} to argue that $X_+u_{m_0}=0$ implies $X_-u_{m_0+2}=0$.
This gives $Xu_k=0$ for all relevant values of~$k$ except $k=m+1$.
Indeed, the fact that also $u_{m+1}=0$ relies on the absence of trace-free conformal Killing tensors, whereas the vanishing of higher degrees does not.

\section{Regularity of integral functions}
\label{sec:reg}

\begin{proof}[Proof of lemma~\ref{lma:u-reg}]
Let us split~$f$ into even and odd parts as $f_\pm=\frac12(f\pm f\circ\reverse)$ with respect to reversion.
We have $f=f_++f_-$ and $f_\pm\circ\reverse=\pm f_\pm$.
We split the function~$u$ similarly into~$u_\pm$.

Consider any broken ray~$\gamma$ with endpoints on~$\outerbdy$.
Since~$u$ vanishes at~$\outerbdy$, it follows from~\eqref{eq:u-def} that~$f$ integrates to zero over~$\gamma$.
Similarly,~$f$ has zero integral over the reverse of the geodesic~$\gamma$.
This implies that both~$f_+$ and~$f_-$ integrate to zero over~$\gamma$.

The integral function of~$f_\pm$ as defined by~\eqref{eq:u-def} is~$u_\mp$.
The functions~$f_\pm$ and~$u_\mp$ satisfy the assumptions of the lemma.
We prove regularity for the two functions~$u_\pm$ separately.
We may thus assume that $u\circ\reverse=\pm u$.

Consider any $(x,v)\in SM$.
The broken rays~$\gamma_{x,\pm v}$ together form a broken ray with endpoints on~$\outerbdy$, and therefore
\begin{equation}
\label{eq:u+u-}
u(x,v)\pm u(x,-v)=0.
\end{equation}
Take any point $(x,v)\in\sisus SM$.
At most one of the two broken rays~$\gamma_{x,\pm v}$ has a tangential reflection because the geometry is admissible.
The boundary is strictly convex, so~$\gamma_{x,\pm v}$ meets~$\outerbdy$ transversally (non-tangentially).
As the boundary components~$\outerbdy$ and~$\innerbdy$ are smooth and the boundary is always met transversally, the broken rays starting near $(x,\pm v)$ depend smoothly on their initial data.
Since $f\in C^2(SM)$, this implies that~$u$ is~$C^2$ in a neighborhood of $(x,\pm v)$.
By~\eqref{eq:u+u-} the function~$u$ is also~$C^2$ in a neighborhood of $(x,v)$.

We have thus shown that $u\in C^2(\sisus SM)$.
If there are tangential reflections (possibly at the initial point)\ntr{Added this note to make it clearer that the difficulty is two-fold: There are issues if the flow meets $\innerbdy$ or $\outerbdy$ tangentially. Due to strict convexity the boundary component $\outerbdy$ can only be reached if we start tangent to it. (In that case the flow will reach the boundary in zero time.)} or $(x,v)$ is tangent to~$\outerbdy$, the broken ray flow is non-smooth but still continuous.
Therefore $u\in C(SM)$.
To show that~$u$ is Lipschitz, it suffices to show that the first order derivatives are uniformly bounded in~$\sisus SM$.

To this end, let $(-\eps,\eps)\ni s\mapsto(x_s,v_s)\in\sisus SM$ be a~$C^1$ unit speed curve on~$SM$ with $(x_0,v_0)=(x,v)$.
We have
\begin{equation}
\label{eq:der}
\begin{split}
\Der{s}u(x_s,v_s)
&=
f(\gamma_{x_s,v_s}(\tau_{x_s,v_s}),\dot\gamma_{x_s,v_s}(\tau_{x_s,v_s}))\Der{s}\tau_{x_s,v_s}
\\&\quad+
\int_0^{\tau_{x_s,v_s}}\Der{s}f(\gamma_{x_s,v_s}(t),\dot\gamma_{x_s,v_s}(t))\der t.
\end{split}
\end{equation}
We wish to show that the derivative~\eqref{eq:der} is bounded at $s=0$, uniformly for all choices of the curve on~$SM$ and the point $(x,v)\in\sisus SM$.
Let $J=\Der{s}\gamma_{x_s,v_s}$ be the Jacobi field corresponding to our variation of the broken ray.

We first study the second (integral) term.
The integrand in~\eqref{eq:der} is bounded by a multiple of $(\abs{J}^2+\abs{\sdot J}^2)^{1/2}\abs{\nabla_{SM}f}$.
Arguing with~\eqref{eq:u+u-} again, we may assume that $\gamma_{x,v}$ contains no reflections with $\abs{\ip{\dot\gamma}{\nu}}<a$.
As shown in~\cite[Corollary 16]{IS:brt}, there is a uniform bound $\abs{J}^2+\abs{\sdot J}^2\leq C_1$ in this case for the same constant $C_1$ for all broken rays on the manifold.
Therefore the integrand is uniformly bounded.

We then turn to the first (boundary) term.
Applying the inverse function theorem to a boundary defining function gives $\Der{s}\tau_{x_s,v_s}|_{s=0}=-\ip{J}{\nu}/\ip{\dot\gamma_{x,v}}{\nu}|_{t=\tau_{x,v}}$.
\ntr{Edited the previous sentence to give more details on the calculation.}
This is uniformly bounded outside any neighborhood of~$\outerbdy$.
It remains to analyze this term for short geodesics which are almost tangent to~$\outerbdy$ and do not reach~$\innerbdy$.

The function $u(x,v)$ vanishes whenever $x\in\outerbdy$.
Therefore $f(x,v)=-Xu(x,v)=0$ whenever $x\in\outerbdy$ and $v\in S_x\outerbdy$; in this case~$Xu$ is a derivative along the boundary.
As the function~$f$ is Lipschitz, we have $\abs{f(x,v)}\leq C_2\abs{\ip{\nu}{v}}$ for all $x\in\outerbdy$ and $v\in S_xM$ for some uniform constant $C_2$.
The derivative $\Der{s}\tau_{x_s,v_s}|_{s=0}$ is bounded uniformly by $C_1/\abs{\ip{\nu}{v}}$, so the first term is bounded by $\aabs{f}_{L^\infty}C_1C_2$.

Both terms in~\eqref{eq:der} are uniformly bounded, and this concludes the proof.
\end{proof}

The estimates obtained above can be improved.
For example, Jacobi fields along almost tangential geodesics are small because the geodesics are short.
This shows that the derivatives at~$\outerbdy$ are not only bounded, but go to zero.

\begin{proof}[Proof of lemma~\ref{lma:uk-reg}]
The function~$u_k$ is obtained from~$u$ by projecting to a fixed spherical harmonic degree fiber by fiber.
It is easy to see that this preserves $C^2$-regularity in the interior and Lipschitz-regularity on the whole~$SM$.
\end{proof}

\section{A Pestov identity with boundary terms}
\label{sec:pestov}

\subsection{The first Pestov identity}

We disregard regularity and symmetry properties at the boundary first.
The resulting first Pestov identity will serve as a stepping stone towards the estimates we need.

\begin{proof}[Proof of lemma~\ref{lma:pestov-smooth}]
We first assume $u\in C^4(SM)$.
Writing the norms as inner products and integrating by parts gives
\begin{equation}
\begin{split}
\norm{\vd X u}^2 - \norm{X\vd u}^2 
&
= \iip{(X \vdiv \vd X - \vdiv X X \vd)u}{u}
\\&\quad
-\iip{\vnu X\vd u}{\vd u}_{\pSM}
-\iip{\vnu\vdiv\vd Xu}{u}_{\pSM}
.
\end{split}
\end{equation}
The commutator formulas can be used to simplify the resulting operator:
\begin{equation}
\begin{split}
X \vdiv \vd X - \vdiv X X \vd
&=
-\hdiv \vd X + \vdiv X \hd
\\&=
- \hdiv \vd X + \vdiv \hd X + \vdiv R \vd
\\&=
-(n-1)X^2+\vdiv R \vd.
\end{split}
\end{equation}
Integrating by parts again leads us to
\begin{equation}
\begin{split}
&
\norm{\vd X u}^2 - \norm{X\vd u}^2 
=
(n-1) \norm{Xu}^2
-\iip{R \vd u}{\vd u}
\\&\qquad
-(n-1)\iip{\vnu Xu}{u}_{\pSM}
-\iip{\vnu X\vd u}{\vd u}_{\pSM}
\\&\qquad
-\iip{\vnu\vdiv\vd Xu}{u}_{\pSM}
.
\end{split}
\end{equation}
The interior terms are as claimed, and it remains to simplify the boundary terms.

To this end we write\ntr{Changed $v$ to $\vnu$ on the second line.}
\begin{equation}
\begin{split}
&-\iip{\vnu\vdiv \vd Xu}{u}_{\pSM}
=\iip{\vd Xu}{\vd(\vnu u)}_{\pSM}
\\&\quad
=\iip{\vd Xu}{\vd(\vnu)u\vnu\vd u}_{\pSM}
\\&\quad
=-\iip{Xu}{\vdiv[\vd(\vnu)u]}_{\pSM}
+\iip{\vd Xu}{\vnu\vd u}_{\pSM}
\\&\quad
=(n-1)\iip{Xu}{\vnu u}_{\pSM}
-\iip{Xu}{\ip{\vd(\vnu)}{\vd u}}_{\pSM}
\\&\qquad
+\iip{\vd Xu}{\vnu\vd u}_{\pSM}.
\end{split}
\end{equation}
Noticing that $\ip{\vd(\vnu)}{\vd u}=\ip{\nu}{\vd u}$, we find that the boundary terms become
\begin{equation}
\begin{split}
&\iip{\vd Xu}{\vnu\vd u}_{\pSM}
-\iip{X\vd u}{\vnu\vd u}_{\pSM}
-\iip{Xu}{\ip{\nu}{\vd u}}_{\pSM}
\\&\quad
=
\iip{(\vd X-X\vd)u}{\vnu\vd u}_{\pSM}
-\iip{\nu Xu}{\vd u}_{\pSM}
\\&\quad
=
\iip{\vnu\hd u-\nu Xu}{\vd u}_{\pSM}
\end{split}
\end{equation}
as desired.
The result for $u\in C^2(SM)$ follows from a simple approximation arguments; all terms in the final identity contain only derivatives up to order two.
\end{proof}

\subsection{Boundary terms with symmetry}
\label{sec:bdy-terms}

To prove lemma~\ref{lma:bdy-term} concerning the boundary terms of the Pestov identity, we first collect some auxiliary results.

Given $(x,v)\in\pSM$, we can represent $T_{(x,v)}\pSM$ in the horizontal and vertical splitting
as $(\xi_{H},\xi_{V})$, where $\xi_{H}\in T_{x}\partial M$ and~$\xi_{V}$ is orthogonal to~$v$.
Then a simple calculation shows that
\begin{equation}
\der\reflect_{(x,v)}(\xi_{H},\xi_{V})
=
(\xi_{H},
\reflect_{x}(\xi_{V})
-2\ip{v,\nabla_{\xi_{H}}\nu}\nu
-2\vnu\nabla_{\xi_{H}}\nu
).
\end{equation}
Let us denote by~$v^{\parallel}$ the orthogonal projection of~$v$ onto~$T_{x}\partial M$ and similarly for other vectors.
Note that since~$\nu$ has norm~$1$, we have $\ip{v}{\nabla_{\xi_{H}}\nu}=\ip{v^{\parallel}}{\nabla_{\xi_{H}}\nu}$, so we can write the formula for~$\der\reflect$ as
\begin{equation}
\der\reflect_{(x,v)}(\xi_{H},\xi_{V})
=
(\xi_{H},
\reflect_{x}(\xi_{V})
-2\ip{v^{\parallel}}{\nabla_{\xi_{H}}\nu}\nu
-2\vnu\nabla_{\xi_{H}}\nu
).
\label{eq:rho}
\end{equation}
It is also useful to compute the adjoint of~$\der\reflect_{(x,v)}$ with respect to the Sasaki metric.
Using that $(a,b)\mapsto\ip{a}{\nabla_{b}\nu}$ is a symmetric form (second fundamental form) we find
\begin{equation}
\der\reflect_{(x,v)}^{*}(\xi_{H},\xi_{V})
=
(
\xi_{H}
-2\ip{\nu}{\xi_{V}}\nabla_{v^{\parallel}}\nu
-2\vnu \nabla_{\xi_{V}^{\parallel}}\nu
,\reflect_{x}(\xi_{V})).
\label{eq:rho*}
\end{equation}
In addition, a simple vertical calculation shows that
\begin{equation}
\label{eq:vrho}
\vd(u\circ\reflect)(x,v)
=
\reflect_{x}((\vd u)\circ\reflect(x,v)).
\end{equation}
We will next move to horizontal derivatives.

We would like to have some insight into the operator $Q=\vnu\ehd-\nu X$.
This operator can be rewritten in such a way that it acts on functions $u\in C^{\infty}(\pSM)$.
This rewriting is important to study the effect of~$\reflect$.

It is convenient to take components of differential operators parallel to~$\partial M$.
We define
\begin{equation}
\label{eq:d-par-def}
\nabla^{\parallel}u
\coloneqq
\nabla u-\ip{\nabla u}{(\nu,0)}(\nu,0).
\end{equation}
and
\begin{equation}
\label{eq:X-par-def}
X^{\parallel}
\coloneqq
X-\ip{X}{(\nu,0)}(\nu,0)
=
(v-\vnu\nu,0)
=
(v^{\parallel},0)
\end{equation}
at the boundary.
\ntr{Moved the definitions of the projected operators before the claim.}

\begin{lemma}
\label{lma:Q}
We have
\begin{equation}
Q
=
\vnu d\pi\nabla^{\parallel}-\nu X^{\parallel}
.
\end{equation}
\end{lemma}

\begin{proof}
The operator~$\ehd$ is just~$\der\pi\nabla$, so applying~$\der\pi$ to the projected operator of~\eqref{eq:d-par-def} we derive
\begin{equation}
\der\pi\nabla^{\parallel}u
=
\ehd u-\ip{\ehd u}{\nu}\nu.
\end{equation}
Similarly, we also project $X=(v,0)$ using~\eqref{eq:X-par-def} and find
\begin{equation}
X^{\parallel}u
=
Xu-\vnu \ip{\ehd u}{\nu}
\end{equation}
and the claim follows.
\ntr{Reworded the proof now that the definitions are before the claim.}
\end{proof}

From this form we can clearly see that~$Q$ acts on ~$C^{\infty}(\pSM)$.
The next two lemmas study the composition of each~$\der\pi\nabla^{\parallel}$ and~$X^{\parallel}$ with~$\reflect$.

\begin{lemma}
\label{lemma:hrho}
We have
\begin{equation}
\der\pi\nabla^{\parallel}(u\circ\reflect)
=
(\der\pi\nabla^{\parallel} u)\circ\reflect
-2\ip{\nu}{(\vd u)\circ\reflect}\nabla_{v^{\parallel}}\nu
-2\vnu\nabla _{[(\vd u)\circ\reflect]^{\parallel}}\nu.
\end{equation}
\end{lemma}

\begin{proof}
The chain rule gives
\begin{equation}
\nabla^{\parallel}(u\circ\reflect)
=
\der\reflect^{*}((\nabla^{\parallel}u)\circ\reflect).
\end{equation}
Since $(\nabla^{\parallel}u)_{V}=\vd u$, formula~\eqref{eq:rho*} proves the lemma.
\end{proof}

\begin{lemma}
\label{lemma:xrho}
We have
\begin{equation}
\begin{split}
X^{\parallel}(u\circ\reflect)
&=
(X^{\parallel}u)\circ\reflect
-2\ip{v^{\parallel}}{\nabla_{v^{\parallel}}\nu}\ip{(\vd u)\circ\reflect}{\nu}
\\&\qquad
-2\vnu\ip{[(\vd u)\circ\reflect]^{\parallel}}{\nabla_{v^{\parallel}}\nu}.
\end{split}
\end{equation}
\end{lemma}

\begin{proof}
The proof is very similar to the previous lemma but one now uses~\eqref{eq:rho}.
\end{proof}

We are now ready to prove the lemma about boundary terms with reflection symmetry.

\begin{proof}[Proof of lemma~\ref{lma:bdy-term}]
We will show that for any $u\in C^{\infty}(\pSM)$ we have
\begin{equation}
\label{eq:PH}
P(u\circ\reflect,u\circ\reflect)
=
-P(u,u)
-2H(u,u)
.
\end{equation}
By simple approximation the same holds for $u\in C^2$
If $u\circ\reflect=\pm u$, then~\eqref{eq:PH} gives $2P(u,u)=-2H(u,u)$, which proves the lemma.

Recall that
\begin{equation}
Tu(x,v)
=
\ip{\vd u}{\nu} v-\vnu\vd u
=
v^{\parallel}\ip{\nu}{\vd u}-\vnu (\vd u)^{\parallel}.
\end{equation}
We begin computing at $(x,v)$ using equation~\eqref{eq:vrho} and lemma~\ref{lma:Q}:
\begin{equation}
\begin{split}
&\ip{Q(u\circ\reflect)}{\vd(u\circ\reflect)}
=
\ip{Q(u\circ\reflect)}{\reflect_{x}((\vd u)\circ\reflect)}
\\
&\qquad=
\ip{\vnu\der\pi(\nabla^{\parallel}(u\circ\reflect))+X^{\parallel}(u\circ\reflect)\nu}{(\vd u)\circ\reflect}
,
\end{split}
\end{equation}
where we also used that $\reflect(\nu)=-\nu$ and~$\reflect$ fixes every vector in~$T_{x}\partial M$.

We now use lemmas~\ref{lemma:hrho} and~\ref{lemma:xrho} to obtain
\begin{equation}
\ip{Q(u\circ\reflect)}{\vd(u\circ\reflect)}
=
-\ip{(Qu)\circ\reflect}{(\vd u)\circ\reflect}+S,
\end{equation}
where $S$ is given by
\begin{equation}
\begin{split}
S
&\coloneqq
-4\vnu\ip{\nu}{(\vd u)\circ\reflect}\Pi\left(v^{\parallel},[(\vd u)\circ\reflect]^{\parallel}\right)
\\&\quad
-2\vnu^2 \Pi\left([(\vd u)\circ\reflect]^{\parallel},[(\vd u)\circ\reflect]^{\parallel}\right)
-2\ip{\nu}{(\vd u)\circ\reflect}^{2}\Pi\left(v^{\parallel},v^{\parallel}\right).
\end{split}
\end{equation}
We can rearrange~$S$ so that
\begin{equation}
\begin{split}
S
=
-2\Pi
\bigg(
&
\ip{\nu}{(\vd u)\circ\reflect} v^{\parallel}
+\vnu[(\vd u)\circ\reflect]^{\parallel},
\\&
\ip{\nu}{(\vd u)\circ\reflect} v^{\parallel}
+\vnu[(\vd u)\circ\reflect]^{\parallel}
\bigg)
.
\end{split}
\end{equation}
Note that
\begin{equation}
S\circ\reflect=\Pi(Tu,Tu).
\end{equation}
To complete the proof of~\eqref{eq:PH} we just need to observe that since~$\reflect_{x}$ is an isometry of each fibre~$S_{x}M$, for any function~$F$ we have
\begin{equation}
\int_{\pSM}F\circ\reflect\der\Sigma^{2n-2}
=
\int_{\pSM}F\der\Sigma^{2n-2}.
\end{equation}
This concludes the proof of~\eqref{eq:PH} and also of the lemma.
\end{proof}

\subsection{Lowering boundary regularity}

We need to be able to apply the Pestov identity in a situation where~$u$ is not~$C^2$ up to~$\pSM$.
We apply an approximation argument.

\begin{proof}[Proof of lemma~\ref{lma:pestov-lip}]
We extend our manifold: Let~$\tilde M$ be a smooth and compact Riemannian manifold with boundary, satisfying $M\subset\sisus\tilde M$.
We can extend~$u$ to a function $S\tilde M\to\R$ satisfying $u\in C^2(\sisus SM)\cap\Lip(S\tilde M)$ and having compact support in $\sisus S\tilde M$.
Let $(u^j)_{j=1}^\infty$ be a sequence of mollifications of~$u$, defined by a smooth partition of unity and the standard convolution method on the Euclidean space via coordinate charts.
We restrict the functions~$u^j$ to~$SM$.

We apply lemma~\ref{lma:pestov-smooth} to~$u^j$, obtaining
\begin{equation}
\label{eq:mollified-pestov}
\begin{split}
\aabs{\vgrad Xu^j}^2_{L^2(SM)}
&=
\aabs{X\vgrad u^j}^2_{L^2(SM)}
-\iip{R\vgrad u^j}{\vgrad u^j}_{L^2(SM)}
\\&\quad
+(n-1)\aabs{Xu^j}^2_{L^2(SM)}
+P_\outerbdy(u^j,u^j)
+P_\innerbdy(u^j,u^j).
\end{split}
\end{equation}
We will study the behaviour of the terms as $j\to\infty$.

Let us look into the boundary terms first.
Recall that $u\in\Lip(\pSM)\subset H^1(\pSM)$.
\ntr{Added a sentence to recall the space $u$ is in over $\pSM$.}
By basic properties of mollifiers, $u^j\to u$ in $\Lip(\pSM)$ and therefore also in~$H^1(\pSM)$.
Since~$u$ vanishes at~$\outerbdy$ and~$P_\outerbdy$ contains only first order derivatives, this implies that $P_\outerbdy(u^j,u^j)\to0$ as $j\to\infty$.

The term at~$\innerbdy$ will not vanish in general.
We split~$u^j$ into even and odd parts with respect to reflection: $u^j_e=\frac12(u^j+u^j\circ\reflect)$ and $u^j_o=\frac12(u^j-u^j\circ\reflect)$, so that $u^j=u^j_e+u^j_o$.
We may write
\begin{equation}
P_\innerbdy(u^j,u^j)
=
P_\innerbdy(u^j_e,u^j_e)
+
P_\innerbdy(u^j_o,u^j_o)
+
P_\innerbdy(u^j_e,u^j_o)
+
P_\innerbdy(u^j_o,u^j_e).
\end{equation}
The limit function satisfies
\begin{equation}
\label{eq:u-bdy-parity}
u\circ\reflect=\pm u
\quad\text{at }
\pSM
\end{equation}
\ntr{Promoted this equation to a numbered one to be able to refer to it.}
by assumption, so the cross terms $P_\innerbdy(u^j_e,u^j_o)$ and $P_\innerbdy(u^j_o,u^j_e)$ vanish in the limit $j\to\infty$, and so does one of the other two terms.

The function~$u$ is even or odd at the boundary depending on the sign of~\eqref{eq:u-bdy-parity}.
\ntr{Added an explanatory sentence and changed the paragraph structure to make the argument clearer.}
Suppose that~$u$ is even at the boundary; the odd case is similar.
By lemma~\ref{lma:bdy-term} we have $P_\innerbdy(u^j_e,u^j_e)=-H_\innerbdy(u^j_e,u^j_e)$.
By the~$H^1$ convergence at the boundary, we get $P_\innerbdy(u^j,u^j)\to-H_\innerbdy(u,u)$.

Similar considerations with~$H^1$ convergence in the interior show that $\aabs{Xu^j}\to\aabs{Xu}$ and $\iip{R\vgrad u^j}{\vgrad u^j}\to\iip{R\vgrad u}{\vgrad u}$.

Let us then turn to the second order terms.
By assumption $\vgrad Xu\in L^2$.
Therefore the mollifications converge in this space: $\vgrad Xu^j\to\vgrad Xu$ in~$L^2$.

We have the commutator formula $X\vgrad u=\vgrad Xu-\hgrad u$.
Both terms are in~$L^2(SM)$, so $X\vgrad u\in L^2$ and we have $L^2$-convergence for the second second order term.

We may now study the limit $j\to\infty$ of the Pestov identity~\eqref{eq:mollified-pestov}.
We have
\begin{equation}
\begin{split}
\aabs{\vgrad Xu^j}^2
&\to
\aabs{\vgrad Xu}^2,
\\
\aabs{X\vgrad u^j}^2
&\to
\aabs{X\vgrad u}^2,
\\
\iip{R\vgrad u^j}{\vgrad u^j}
&\to
\iip{R\vgrad u}{\vgrad u},
\\
\aabs{Xu^j}^2
&\to
\aabs{Xu}^2,
\\
P_\outerbdy(u^j,u^j)
&\to
0,\quad\text{and}
\\
P_\innerbdy(u^j,u^j)
&\to
P_\innerbdy(u,u).
\end{split}
\end{equation}
This gives the first claim.

For the second one, observe that by concavity of~$\innerbdy$ we have $H_\innerbdy(u,u)\leq0$ and by non-positive sectional curvature $\ip{Rw}{w}\leq0$.
This gives the second claim.
\end{proof}

The approximation argument used in~\cite{IS:brt} was based on shrinking the manifold instead of mollifying the functions.
When working with functions of a fixed degree instead of the whole~$u$, we find the mollification argument more tractable.

\section{Properties of $X_\pm$}
\label{sec:X+-}

\subsection{$L^2$ estimates for derivatives on the sphere bundle}

We begin by estimating~$X_\pm$ in terms of the other horizontal derivatives~$X$ and~$\hgrad$.

\begin{lemma}
\label{lma:horizontal-L2}
Let~$M$ be a complete and smooth Riemannian manifold, compact or non-compact, with or without boundary.
If $u\in H^1(SM)$, then $X_\pm u\in L^2(SM)$ and
\begin{equation}
\aabs{X_+u}^2
+
\aabs{X_-u}^2
\leq
\aabs{Xu}^2
+
\aabs{\hgrad u}^2.
\end{equation}
\end{lemma}

\begin{proof}
The proof can be found in~\cite[lemma 5.1]{LRS:CH}, which in turn is based on~\cite[lemmas 3.3, 3.5, and 4.4]{PSU:Beurling}.
None of these arguments rely on special geometric assumptions on the manifold.
\end{proof}

We point out that $\aabs{Xu}^2+\aabs{\hgrad u}^2=\aabs{\ehd u}^2$.

\begin{proof}[Proof of lemma~\ref{lma:X+-u-L2}]
By lemma~\ref{lma:u-reg} we have $u\in\Lip(SM)\subset H^1(SM)$.
Lemma~\ref{lma:horizontal-L2} then shows that $X_\pm u\in L^2(SM)$.
\end{proof}

\begin{proof}[Proof of lemma~\ref{lma:L2-reg}]
Since $f\in C^2(SM)$, the transport equation $Xu=-f$ gives $\vgrad Xu=-\vgrad f\in C^1(SM)\subset L^2(SM)$.

Fix then any $k\geq0$.
By lemma~\ref{lma:X+-u-L2} we have $X_\pm u\in L^2(SM)$ and therefore $(X_\pm u)_l\in L^2(SM)$ for any $l\geq0$.
Thus
\begin{equation}
Xu_k
=
(X_++X_-)u_k
=
(X_+u)_{k+1}
+
(X_-u)_{k-1}
\in
L^2(SM).
\end{equation}
For $(X_\pm u)_{k\pm1}$ we have (cf.~\eqref{eq:vgrad-uk})
\begin{equation}
\aabs{\vgrad(X_\pm u(x,v))_{k\pm1}}^2
=
(k\pm1)(k\pm1+n-2)
\aabs{(X_\pm u(x,v))_{k\pm1}}^2,
\end{equation}
whence $\vgrad(X_\pm u(x,v))_{k\pm1}\in L^2(SM)$ and so $\vgrad Xu_k\in L^2(SM)$.
\end{proof}

\subsection{The transport equation}

The next proof is based on the transport equation $Xu=-f$.
The proof is elementary but we record it here for clarity.

\begin{proof}[Proof of lemma~\ref{lma:X+=X-}]
Projecting the transport equation to degree $k+1$ gives
\begin{equation}
X_+u_k+X_-u_{k+2}=-f_{k+1}.
\end{equation}
Since~$f$ is a tensor field of order~$m$, we know that $f_l=0$ when $l>m$ or~$l$ and~$m$ have different parity.
Thus by the assumption of the lemma $f_{k+1}=0$ and the claim follows.
\end{proof}

\subsection{An estimate for the Beurling transform}

If~$X_-$ is surjective, then for any $f_k\in C^\infty(SM)$ of degree~$k$ there is a unique function $f_{k+2}\in C^\infty(SM)$ of degree $k+2$ so that $X_-f_{k+2}=-X_+f_k$ and~$f_{k+2}$ is orthogonal to the kernel of~$X_-$.
The corresponding mapping $f_k\mapsto f_{k+2}$ is called the Beurling transform.
We, however, do not need this transform.
The estimate we get below amounts to a continuity estimate for the Beurling transform but we have no need to formalize the transform itself in the present context.
For a detailed analysis of the Beurling transform and its use in tensor tomography, see~\cite{PSU:Beurling}.

\begin{proof}[Proof of lemma~\ref{lma:X-<X+}]
This proof is analogous to that of~\cite[Proposition 3.4]{PSU:Beurling}.
We apply lemma~\ref{lma:pestov-lip} to~$u_k$; this function satisfies the assumptions by lemmas~\ref{lma:uk-reg} and~\ref{lma:L2-reg}.
The function~$u$ vanishes at~$\outerbdy$ since~$f$ is in the kernel of the broken ray transform, so also~$u_k$ vanishes at~$\outerbdy$.
Estimate~\eqref{eq:pestov-estimate} gives
\begin{equation}
\label{eq:vv1}
\aabs{\vgrad Xu_k}^2
\geq
\aabs{X\vgrad u_k}^2
+(n-1)\aabs{Xu_k}^2.
\end{equation}
Using commutator formulas and vertical eigenvalue properties gives (cf.~\cite[Proof of Proposition 3.4]{PSU:Beurling})
\begin{equation}
\label{eq:vv2}
\begin{split}
\aabs{X\vgrad u_k}^2
&=
\aabs{\vgrad Xu_k}^2
-(n-1)\aabs{Xu_k}^2
+\aabs{\hgrad u_k}^2
\\&\quad
-(2k+n-1)\aabs{X_+u_k}^2
+(2k+n-3)\aabs{X_-u_k}^2
.
\end{split}
\end{equation}
Combining~\eqref{eq:vv1} and~\eqref{eq:vv2} gives
\begin{equation}
(2k+n-1)\aabs{X_+u_k}^2
\geq
\aabs{\hgrad u_k}^2
+(2k+n-3)\aabs{X_-u_k}^2.
\end{equation}
Using $\aabs{\hgrad u_k}^2\geq0$ gives the claimed estimate since $C(k,n)=\frac{2k+n-1}{2k+n-3}$.
\end{proof}

The constant~$C(k,n)$ in the estimate above is not sharp.
However, it is sufficient for us, so we trade optimality for convenience.
Sharper bounds for the Beurling transform can be found in~\cite{PSU:Beurling}.
With the better bounds the products $B(k,N,n)$ are uniformly bounded, but lemma~\ref{lma:prod-est} is strong enough for our theorem.

\subsection{Injectivity of $X_+$ and trace-free conformal Killing tensors}

Lemma~\ref{lma:X+inj} concerns injectivity of the operator~$X_+$.
It was only stated for $k\geq3$, as that was all we needed for the proof of theorem~\ref{thm:tensor}.
The proof relies on properties of conformal Killing tensors.
We give a new proof of the required result in proposition~\ref{prop:killing1}, making our proof of theorem~\ref{thm:tensor} more self-contained.
See section~\ref{sec:killing} and especially~\cite{DS:Killing} for more details on conformal Killing tensor fields.

\begin{proof}[Proof of lemma~\ref{lma:X+inj}]
If~$u$ has degree~$k$ and $X_+u=0$, then~$u$ is a trace-free conformal Killing tensor of rank~$k$.
Ellipticity of~$X_+$ was proven in~\cite{GK:n}, and it follows that such tensor fields are smooth (as also observed in~\cite{DS:Killing}).
By~\cite[theorem 1.3]{DS:Killing} any trace-free conformal Killing tensor vanishing on an open subset of the boundary has to vanish identically.
\end{proof}

\subsection{An estimate for products of constants}

Iterating lemmas~\ref{lma:X+=X-} and~\ref{lma:X-<X+}, we end up with a product of the constants $C(k,n)$.
We therefore need an estimate for this product.

\begin{proof}[Proof of lemma~\ref{lma:prod-est}]
Using $\log(1+x)\leq x$ for $x>0$ we find
\begin{equation}
\label{eq:vv3}
\begin{split}
B(k,N,n)
&=
\prod_{l=1}^N C(k+2l,n)
\\&=
\exp\left(
\sum_{l=1}^N \log(C(k+2l,n))
\right)
\\&\leq
\exp\left(
\sum_{l=1}^N \frac{2}{2(k+2l)+n-3}
\right)
\\&=
\exp\left(
\frac12\sum_{l=1}^N \frac{1}{l+(2k+n-3)/4}
\right).
\end{split}
\end{equation}
A simple comparison of series and integrals gives
\begin{equation}
\sum_{l=1}^N \frac{1}{l+\alpha}
\leq
\int_0^N \frac{\der x}{x+\alpha}
=
\log(1+N/\alpha)
\end{equation}
for any $\alpha>0$.
Using this with~\eqref{eq:vv3} gives the claimed estimate.
\end{proof}

\section{An alternative proof}
\label{sec:alt}

\subsection{Outline of proof}

As before, we begin by presenting the lemmas we need to prove the theorem.
The lemmas will be proved in section~\ref{sec:alt-lma-pf}.

\begin{lemma}
\label{lma:vP}
Let $(M,g)$ be a compact Riemannian manifold with or without boundary.
If $u\in C^2(\sisus SM)\cap\Lip(SM)$, $Xu$ has finite degree, and $\vgrad Xu,X\vgrad u\in L^2(SM)$, then
\begin{equation}
\aabs{X\vgrad u}^2
-
\aabs{\vgrad Xu}^2
+
(n-1)\aabs{Xu}^2
=
\iip{Xu}{[X,\Delta]u}
+
\aabs{\hgrad u}^2.
\end{equation}
\end{lemma}

Combining lemmas~\ref{lma:vP} and~\ref{lma:pestov-lip} gives a convenient identity.

\begin{lemma}
\label{lma:combined-pestov}
Let $(M,g,\outerbdy)$ be admissible.
If $u\in C^2(\sisus SM)\cap\Lip(SM)$, $\vgrad Xu\in L^2(SM)$, and $u=u\circ\reflect$ at $\innerbdy$ and $u=0$ at~$\outerbdy$, then
\begin{equation}
\iip{Xu}{[X,\Delta]u}
=
\iip{R\vgrad u}{\vgrad u}
-\aabs{\hgrad u}^2
+H_\innerbdy(u,u).
\end{equation}
\end{lemma}

The final missing piece is projecting a commutator to degree $m$.

\begin{lemma}
\label{lma:comm-proj}
If $u\in C^2(\sisus SM)$ satisfies $u_{m-1}=0$,
\ntr{Changed the condition. The assumption was ``has degree $m+1$'', but with $u_{m-1}=0$ the claim and particularly its use is clearer.}
then
\begin{equation}
([X,\Delta]u)_m
=
(2m+n-1)(Xu)_m.
\end{equation}
\end{lemma}

\begin{proof}[Second proof of theorem~\ref{thm:tensor}]
We define the function $u$ as in the first proof and observe that by construction $u=u\circ\reflect$ at~$\innerbdy$.
\ntr{Added this observation of reflection property here, as it is used in proposition~\ref{prop:killing1}.}
We define $w\colon SM\to\R$ by
\begin{equation}
w
=
\sum_{k\geq m}u_k
=
u_{m+1}+u_{m+3}+\dots,
\end{equation}
where we have used the fact that every other degree term of~$u$ vanishes by parity considerations.
The goal is to show that $w=0$.

The transport equation $Xu=-f$ gives $(Xw)_k=-f_k=0$ for $k>m$.
Since~$w$ only contains degrees $m+1$ and higher, $(Xw)_k=0$ for $k<m$.
The remaining term is $(Xw)_m=-f_m-X_{+}u_{m-1}\eqqcolon-g_m$, whence $Xw=-g_m$.
\ntr{Proof adjusted slightly to include $Xw=-f_m-X_+u_{m-1}$. It is enough that this is of order $m$.}

The functions~$u$ and $u_{m-1},u_{m-3},\dots$ have regularity properties due to lemmas~\ref{lma:u-reg}, \ref{lma:uk-reg}, and~\ref{lma:L2-reg}.
Therefore $w\in C^2(\sisus SM)\cap\Lip(SM)$, $\vgrad Xw\in L^2(SM)$, $w=w\circ\reflect$ at $\innerbdy$, and $w=0$ at~$\outerbdy$.
Applying lemma~\ref{lma:combined-pestov} to~$w$ and using signs of curvature, we obtain
\begin{equation}
\iip{Xw}{[X,\Delta]w}
\leq
0.
\end{equation}
But $Xw=-g_m$ has only degree~$m$ and the different degrees are orthogonal, so
\begin{equation}
\iip{Xw}{[X,\Delta]w}
=
-\iip{g_m}{([X,\Delta]w)_m}.
\end{equation}
Lemma~\ref{lma:comm-proj} gives us $([X,\Delta]w)_m=-(2m+n-1)g_m$ since $w_{m-1}=0$ and~$w_k$ for\ntr{Changed $n$ to $m$.} $k\geq m+2$ does not affect $([X,\Delta]w)_m$.
Thus
\begin{equation}
(2m+n-1)\aabs{g_m}^2
\leq
0.
\end{equation}
Therefore $Xw=-g_m=0$.

The function~$w$ is invariant under the geodesic flow and reflections at~$\innerbdy$, so it is constant along every broken ray.
It vanishes at~$\outerbdy$, so by admissibility of the geometry $w=0$.
\end{proof}

In the proof we applied the Pestov identity to the function~$w$, not~$u$ or~$u_k$.

\subsection{Proofs of lemmas}
\label{sec:alt-lma-pf}

To complete the second proof of theorem~\ref{thm:tensor}, we prove lemmas~\ref{lma:vP}, \ref{lma:combined-pestov}, and~\ref{lma:comm-proj}.

\begin{proof}[Proof of lemma~\ref{lma:vP}]
We prove the lemma on a shrinked manifold $M^\eps\subset\sisus M$ for which~$\partial M^\eps$ is at distance $\eps>0$ from~$\partial M$ at every point.
By assumption we have $X\vgrad u,\vgrad Xu,Xu,\hgrad u\in L^2(SM)$.
Since $\aabs{F}_{L^2(SM^\eps)}\to\aabs{F}_{L^2(SM)}$ for any $F\in L^2(SM)$, the limit is well-behaved and the inner product $\iip{Xu}{[X,\Delta]u}$ must also exist on the whole sphere bundle~$SM$.

Since~$Xu$ has finite degree, so does $[X,\Delta]u$.
As argued in the proof of lemma~\ref{lma:L2-reg}, vertical derivatives do not change integrability and it follows that $[X,\Delta]u\in L^2(SM^\eps)$.
Thus by a simple approximation argument it is enough to prove the statement for smooth functions~$u$.

We thus assume $u\in C^\infty(SM)$.
Using a commutator formula and vertical integration\ntr{Fixed a spelling error.} by parts gives
\begin{equation}
\aabs{X\vgrad u}^2
=
\aabs{\vgrad Xu-\hgrad u}^2
=
\aabs{\vgrad Xu}^2
+
\aabs{\hgrad u}^2
+2\iip{Xu}{\vdiv\hgrad u}
.
\end{equation}
Applying the commutator formula~\eqref{eq:[X,D]} then gives the desired identity.
\end{proof}

The obtained identity is similar to~\eqref{eq:vv2} and~\cite[Proof of propsition 3.4]{PSU:Beurling}.

\begin{proof}[Proof of lemma~\ref{lma:combined-pestov}]
As stated above, this follows by combining lemmas~\ref{lma:pestov-lip} and~\ref{lma:vP}.
\end{proof}

\begin{proof}[Proof of lemma~\ref{lma:comm-proj}]
The commutator $[X,\Delta]$ is a second order differential operator, so the claimed identity is preserved under~$C^2$ limits locally in the base and globally in the fiber.
Therefore it suffices to prove the statement for $u\in C^\infty(SM)$.

Since $u_{m-1}=0$, we have
\begin{equation}
([X,\Delta]u)_m
=
X_-\Delta u_{m+1}
-\Delta X_-u_{m+1}.
\end{equation}
Functions of degree~$k$ are eigenfunctions of~$\Delta$ with eigenvalue $k(k+n-2)$, so
\begin{equation}
\begin{split}
\Delta u_{m+1}&=(m+1)(m+n-1)u_{m+1},
\quad\text{and}\\
\Delta X_-u_{m+1}&=m(m+n-2)X_-u_{m+1}.
\end{split}
\end{equation}
This leads to
\begin{equation}
([X,\Delta]u)_m
=
(2m+n-1)X_-u_{m+1}.
\end{equation}
Since $u_{m-1}=0$, this is the claimed identity.
\end{proof}

\subsection{Conformal Killing tensors}
\label{sec:killing}

The second method of proof can also be used to study conformal Killing tensor fields.

A tensor field of rank~$m$ may be regarded as function on the sphere bundle containing only degrees~$m$, $m-2$, and so on.
The tensor field is called trace-free if it only contains the top degree, that is, if the corresponding function on~$SM$ has degree~$m$.
As mentioned in the proof lemma~\ref{lma:X+inj}, via this identification, a trace-free conformal Killing tensor field of rank~$m$ is a function $u\in C^\infty(SM)$ of degree~$m$ for which $X_+u=0$.
That is, the trace-free conformal Killing tensors constitute precisely the kernel of the operator~$X_+$.
For details, we refer the reader to~\cite{GK:n,DS:Killing}.

We needed an injectivity result for~$X_+$, and this was stated in lemma~\ref{lma:X+inj}.
Using the methods of the second proof of our main result, we present an alternative proof of injectivity of~$X_+$.
The following proposition could be used as a substitute for lemma~\ref{lma:X+inj} in our first proof of theorem~\ref{thm:tensor}.

\begin{proposition}
\label{prop:killing1}
Assume that $(m,g,\outerbdy)$ is admissible.
Suppose $u\in C^2(\sisus SM)\cap\Lip(SM)$ has degree~$m$, and satisfies $u=0$ at $\outerbdy$ and $u=u\circ\reflect$ at~$\innerbdy$.
If $X_+u=0$, then $u=0$.

In other words, there are no non-trivial trace-free conformal Killing tensors satisfying these boundary conditions at~$\outerbdy$ and~$\innerbdy$.
\end{proposition}

\begin{proof}
Assume first that $2m+n-3>0$.

As argued in the proof of lemma~\ref{lma:L2-reg}, it follows from the given regularity assumptions and having a single degree that $\vgrad Xu\in L^2(SM)$.
Hence lemmas~\ref{lma:pestov-lip} and~\ref{lma:vP} are available, and therefore so is lemma~\ref{lma:combined-pestov}.
We find
\begin{equation}
\iip{Xu}{[X,\Delta]u}
\leq
0.
\end{equation}
As $X_+u=0$, we have $Xu=X_-u\in\Lambda^{m-1}$.
We can thus apply lemma~\ref{lma:comm-proj} with $m$ replaced by $m-1$ to obtain
\begin{equation}
([X,\Delta]u)_{m-1}
=
(2m+n-3)Xu
\end{equation}
and so
\begin{equation}
\iip{Xu}{[X,\Delta]u}
=
(2m+n-3)\aabs{Xu}^2
\leq
0,
\end{equation}
implying $Xu=0$.
Now~$u$ is constant along every broken ray and vanishes at~$\outerbdy$, so $u=0$ as claimed.

If $2m+n-3\leq0$, then $m=0$ and $n=2,3$.
Therefore we then consider the case $m=0$.
Now $X_-u=0$, and so $Xu=X_+u=0$, and by the same argument $u=0$.
\end{proof}

A similar proof also provides a result without reflections.
The result is not new (cf.~\cite{DS:Killing}), but we record it for the sake of having a new and simple proof.

\begin{proposition}
\label{prop:killing2}
Let~$M$ be a compact Riemannian manifold with boundary.
Let $\outerbdy\subset\partial M$ be such that any point in~$M$ can be reached by a geodesic that meets~$\outerbdy$ transversally.
If a trace-free conformal Killing tensor $u\in C^2(\sisus M)\cap\Lip(M)$ of any order $m\geq0$ vanishes at~$\outerbdy$, then $u=0$.
\end{proposition}

The condition on~$\outerbdy$ is always met if $\outerbdy=\partial M$, as any point can be connected to the boundary by a shortest geodesic and it is normal to the boundary.

\begin{proof}[Proof of proposition~\ref{prop:killing2}]
By the arguments of the previous proof we find $Xu=0$, so that~$u$ is invariant under the geodesic flow.
Fix any $x\in M$.
There is a direction $v_0\in S_xM$ so that the geodesic starting at $(x,v_0)$ meets~$\outerbdy$ transversally.
By the implicit function theorem there is a neighborhood $U\subset S_xM$ of~$v_0$ so that the geodesic starting at $(x,v)$ with $v\in U$ reaches~$\outerbdy$ in finite time.

As~$u$ is invariant under the geodesic flow and vanishes at~$\partial M$, it follows that $u(x,v)=0$ for all $v\in U$.
Since~$u$ is of finite degree, it is analytic in~$v$ and so $u(x,v)=0$ for all $v\in S_xM$.
This holds for all $x\in M$, so we may conclude that $u=0$.
\end{proof}

\section*{Acknowledgements}
Much of the research leading to this paper was completed during the inverse problems program at Institut Henri Poincar\'e in 2015 and during the first author's visits to Cambridge.
The authors thank CNRS and IHP for support and hospitality.
J.I.\ was supported by the Academy of Finland (decision 295853) and an encouragement grant from the Emil Aaltonen foundation, and he is grateful for all the hospitality at Cambridge.
G.P.P.\ was supported by an EPSRC grant (EP/R001898/1).
We are grateful to the anonymous referees for useful feedback.
\ntr{Added thanks to referees. The comments were very useful!}

\bibliographystyle{plain}

\end{document}